%% file: mainTchiNeighbourhood.tex
\documentclass{jgaa-art}

\input{config/packages}

\input{config/definitions}

\begin{document}

% --------------------------------------------------------------------
\doi{10.7155/jgaa.00472} % doi number
% --------------------------------------------------------------------
\Issue{22}{2}{329}{355}{2018} % volume, number, start page, end page, year
% --------------------------------------------------------------------
\HeadingAuthor{Klawitter}
\HeadingTitle{The SNPR neighbourhood of tree-child networks}
% --------------------------------------------------------------------
\title{The SNPR neighbourhood of tree-child networks} % Paper title
\Ack{Research supported by New Zealand Marsden Fund} % Grants Acknowledgments
% --------------------------------------------------------------------
\author[]{Jonathan Klawitter}{jo.klawitter@gmail.com} % Authors
\affiliation[]{Department of Computer Science, University of Auckland, New Zealand} % Affiliations
% --------------------------------------------------------------------

%% --------------------------------------------------------------------
%   Paper history
%% --------------------------------------------------------------------
\submitted{July 2017}%
\reviewed{January 2018}%
\revised{January 2018}%
\reviewed{April 2018}%
\revised{April 2018}%
\accepted{July 2018}%
\final{July 2018}%
\published{August 2018}% 
\type{Regular Paper}%
\editor{F. Vandin}%
%% --------------------------------------------------------------------
 
\maketitle 

%% --------------------------------------------------------------------
%       Abstract
%% --------------------------------------------------------------------
\begin{abstract}
Network rearrangement operations like SNPR (SubNet Prune and Regraft), a recent generalisation of
rSPR (rooted Subtree Prune and Regraft), induce a metric on phylogenetic networks. To search the
space of these networks one important property of these metrics is the sizes of the neighbourhoods,
that is, the number of networks reachable by exactly one operation from a given network. In this
paper, we present exact expressions for the SNPR neighbourhood of tree-child networks, which depend
on both the size and the topology of a network. We furthermore give upper and lower bounds for
the minimum and maximum size of such a neighbourhood.
\end{abstract}

% --------------------------------------------------------------------
\Body
%% --------------------------------------------------------------------

\input{sections/introduction}
\input{sections/preliminaries}
\input{sections/tchiSNPR}
\input{sections/extremeValues}
\input{sections/discussion}

\section*{Acknowledgements}
I would like to thank the New Zealand Marsden Fund for their financial support, Simone Linz for
spirited discussions, guidance and helpful comments, and the reviewers for their helpful comments
and suggestions.

%% --------------------------------------------------------------------
%       Bibliography
%% --------------------------------------------------------------------
\clearpage

\bibliography{sources}
\bibliographystyle{abbrv}

\end{document}

%% file: config/packages.tex
\usepackage[utf8]{inputenc}
\usepackage[T1]{fontenc}
\usepackage{amsmath,amssymb}

\usepackage{enumitem}

\DeclareGraphicsExtensions{.pdf,.png,.eps}
\graphicspath{{figs/}} 

\usepackage{cite}
\usepackage[capitalise,noabbrev,nameinlink]{cleveref}
\crefname{enumi}{Case}{Cases}
\newcommand{\etal}{{et~al.}}

\usepackage{xspace}

%% file: config/definitions.tex
\newtheorem{proposition}{Proposition}
\newtheorem{corollary}{Corollary}
\newcommand{\qedd}{\hspace*{\fill}$\Box$\ifmmode\else\par\addvspace\topsep\fi}
\newenvironment{proofof}[1][]{\par\addvspace\topsep\noindent{\bf Proof of #1:} \ignorespaces
}{\qedd}

\usepackage{mathtools}
\DeclarePairedDelimiter\abs{\lvert}{\rvert}
\def\nats{\ensuremath{\mathbb{N}}}
\def\Oh{\ensuremath{\mathcal{O}}}

\def\root{\ensuremath{\rho}} % for root
\def\trees{\ensuremath{\mathcal{T}_n}} % for trees
\def\treesNoN{\ensuremath{\mathcal{T}}} % for trees
\def\nets{\ensuremath{\mathcal{N}_n}} % for networks
\def\rphytrees{rooted phylogenetic trees\xspace} 
\def\rphynets{rooted phylogenetic networks\xspace} 
\def\phytree{phylogenetic tree\xspace} 
\def\phytrees{phylogenetic trees\xspace} 
\def\phynet{phylogenetic network\xspace} 
\def\phynets{phylogenetic networks\xspace}
\def\tchi{tree-child\xspace}
\def\tchinets{\ensuremath{\mathcal{TC}_n}}
\def\tchinetsNoN{\ensuremath{\mathcal{TC}}}
\def\tsib{tree-sibling\xspace}
\def\tbased{tree-based\xspace}
\def\retvis{reticulation-visible\xspace}

\def\trias{\ensuremath{t_3}} % number of triangles
\def\treebranchtriangles{\ensuremath{t_3^{*}}} % number of triangles with pure tree branches
\def\treebottomedges{\ensuremath{E_{\bar{\trias}}}}
\def\dias{\ensuremath{d_4}} % number of diamonds
\def\trapis{\ensuremath{t_4}} % number of trapezoids

\newcommand{\OPSNPR}{\textup{SNPR}\xspace}
\newcommand{\SNPR}{\textup{\texttt{SNPR}}\xspace}
\newcommand{\SNPRP}{\textup{\texttt{SNPR$+$}}\xspace}
\newcommand{\SNPRM}{\textup{\texttt{SNPR$-$}}\xspace}
\newcommand{\NNI}{\textup{\texttt{NNI}}\xspace}
\newcommand{\NNIP}{\textup{\texttt{NNI$+$}}\xspace}
\newcommand{\NNIM}{\textup{\texttt{NNI$-$}}\xspace}

%% file: sections/introduction.tex
\section{Introduction}
\label{sec:introduction}

% # background context
Phylogenetic trees and networks are used to represent and study the evolutionary relationships of
species and languages. The set of all hypothesised networks to model the relationships for a set
of data is referred to as a space of phylogenetic networks.
To navigate and work with such a space one common tool is using rearrangement operations that
transform one network into another one and thereby induce a metric on the space. 
For phylogenetic trees three well known such operations are the Nearest Neighbour Interchange
(NNI)~\cite{Rob71}, the Subtree Prune and Regraft (SPR) and Tree Bisection and Reconnection
(TBR)~\cite{AS01} operations.
In recent years, the study of rearrangement operations has moved from the tree space to the network
space. 
%		NNI
For example, the NNI operation has been generalised from trees to networks by Huber
\etal~\cite{HLMW16}, and further considered in its unrooted variant~\cite{HMW16,FHMW17} 
and its rooted variant~\cite{GvIJLPS17,JJEvIS17}.
%		SPR
Francis \etal~\cite{FHMW17} also introduced SPR and TBR on unrooted \phynets.
Furthermore and as with this paper, the generalisation of rSPR (rooted SPR) operation for
\rphynets has been studied. 
Bordewich \etal~\cite{BLS16} introduced the SubNet Prune and Regraft (SNPR) operation, which
allows to navigate between \phynets with both the same and with different number of reticulations.
Gambette \etal~\cite{GvIJLPS17} and Janssen \etal~\cite{JJEvIS17} studied a slightly more powerful
rSPR operation (reusing the original name) for \phynets with the same number of reticulations.

% # problem and its history 
An interesting property of these operations and spaces is the sizes of a neighbourhood of a
network. The neighbourhood problem with respect to a type of operation asks how many
networks in the space are exactly one such operation apart from a given network. 
Robinson~\cite{Rob71} already considered this question when he laid the foundation for the studies
of rearrangement operations with the introduction of NNI on unrooted phylogenetic trees.
Allen and Steel~\cite{AS01} solved the problem for SPR on unrooted phylogenetic trees.
The size of the neighbourhood of an unrooted phylogenetic tree, for both NNI and SPR, only depends
on the number of leaves of the tree. However, Allen and Steel~\cite{AS01} further showed that for
unrooted trees and TBR, the size of the neighbourhood depends not only on the number of leaves but
also on the topology of the tree.
Humphries and Wu~\cite{HW13} later gave a closed formula for the neighbourhood under TBR using the
non-trivial splits of a tree to represent the topology.
Beyond that, Baskowski \etal~\cite{BMSW15} considered the problem for SPR and TBR on unrooted
\phytrees that are restricted to a circular ordering of its leaves, and de~Jong \etal~\cite{DJMLS16}
considered the problem of finding neighbours that are two or more operations away for an unrooted
\phytree.
For \rphytrees and rSPR the size of the neighbourhood depends again on the number of leaves,
but also on the topology of the tree. Song~\cite{Son03} gave a formula for this problem where he
characterises the size by the number of ancestors of each vertex in a tree. He used a recursive
approach to count neighbours, which he then transformed into a closed formula. Furthermore,
Song~\cite{Son06} did the same for totally ordered \phytrees.

The problem of determining the neighbourhood size gets harder for \phynets. Huber
\etal~\cite{HLMW16} solved the problem for NNI operations on unrooted level-1 networks. They showed
that the size depends on structures of the network like the number of cycles of size three and four.
Gambette \etal~\cite{GvIJLPS17} extended this with an upper bound for the neighbourhood size of
an unrooted \phynet and NNI, but restricted to networks with a fixed number of reticulations.
Similarly, Francis \etal~\cite{FHMW17} gave an upper bound for the same networks but for SPR instead
of for NNI.
% # contribution
In this paper, we will show that for \rphynets and SNPR
the neighbourhood size depends on the number of leaves, the topology of the network in terms of
descendants and occurrences of certain structures.
We will outline why identifying these structures is difficult for \rphynets in general and for
the classes of \tbased~\cite{FS15} and \retvis~\cite{HK07} networks. However, for classes like
\tchi~\cite{Wil07}, normal~\cite{Wil10} and level-1~\cite{GBP09} networks, the dependencies of the
neighbourhood size on the topology are comprehensible. 
We will focus on the class of tree-child networks, a class for which several problems that are
difficult in general can be solved efficiently~\cite{vISS10,CLS14,BS16,BST17}.
The class of tree-child networks is also not as restricted as normal and level-1 networks, which are
in fact subclasses of it.
The main result of this paper is a formula for the neighbourhood size of a tree-child network under
SNPR. As byproduct we reprove the formula for \rphytrees by Song.
We also give bounds on the neighbourhood size that only depend on the number of leaves
(\cref{sec:minMaxValues}).
First, however, we introduce the notation and terminology used throughout this paper.

%% file: sections/preliminaries.tex
\section{Preliminaries}
\label{sec:preliminaries}

We now recall the definitions of \rphynets, the class of tree-child networks, and the SNPR and NNI
operations. We also formally define the unit neighbourhood problem and structures of a \phynet like
triangles, diamonds and trapezoids.

\paragraph{Phylogenetic networks.}
A \emph{rooted binary \phynet} $N = (V, E)$ is a directed acyclic graph with edges $E$ and the
following vertices $V$:
\begin{itemize}
  \item the \textit{root} $\root$ with in-degree zero and out-degree one,
  \item $n$ \emph{leaves} with in-degree one and out-degree zero bijectively labelled with a set of taxa,
  \item \emph{inner tree vertices} with in-degree one and out-degree two, and
  \item \emph{reticulations} with in-degree two and out-degree one. 
\end{itemize}
The \emph{tree vertices} of $N$ are the union of the inner tree vertices, the leaves and the root.
The unique edge incident to the root is called the \emph{root edge} $e_\root$. 
An edge $e = (u, v)$ is called a \emph{reticulation edge}, if $v$ is a
reticulation, and a \emph{tree edge}, if $v$ is a tree vertex. An edge not incident to the root or
a leaf is an \emph{inner} edge.
Furthermore, we call $e = (u, v)$ \emph{pure}, if $u$ and $v$ are both either tree vertices or
reticulations, and \emph{impure} otherwise.
Throughout this paper we assume that $n \geq 2$ and let $r$ denote the number of
reticulations. There are thus $m = 2n + 3r - 1$ edges in $N$ \cite[Lemma 2.1]{MSW15}. 

Our definition of a rooted binary \phynet allows the existence of parallel edges. 
Furthermore, we note that our definition of the root is known as \emph{pendant root} \cite{BLS16}
and that it differs from another common definition where the root has out-degree two. Our
variation serves both elegance and technical reasons. 
Since we only consider a fixed set of taxa, we omit its notation.
Moreover, throughout this paper we only consider phylogenetic networks that are both rooted and
binary and therefore refer to them simply as \phynets.

Let $N$ be a \phynet. For two vertices $u$ and $v$ in $N$, we say that $u$ is a \emph{parent} of
$v$ and $v$ is a \emph{child} of $u$, if there is an edge $(u, v)$ in $N$.
We say $u$ is \emph{ancestor} of $v$ and $v$ is \emph{descendant} of $u$ if there is a directed path
from $u$ to $v$ in $N$. 
We say $u$ and $v$ are \emph{siblings} if they have a common parent. 
For ease of use, we also say that $u$ is an \emph{uncle} of $v$, if $u$ is
sibling of a parent of $v$. In reverse, $v$ is then the \emph{nephew} of $u$.

Let $(u, v), (x, y)$ be edges of $N$.
We say $(x, y)$ is a parent edge of $(u, v)$ if $y = u$. Consequently, $(u, v)$ is then a child edge
of $(x, y)$. The two edges are considered siblings if $u = x$.
We say $(u, v)$ is a descendant of $(x, y)$ if either $y = u$  or if $u$ is a descendant of $y$. In
return, $(x, y)$ is then an ancestor of $(u, v)$.
For an edge $e$ of $N = (V, E)$ we use the function $\delta \colon E \to \nats$ to count the number
of descendant edges of $e$, i.e.
$\delta(e) := \abs{ \{f \in E \mid f \text{ is descendant of } e\} }$. For example, in the
phylogenetic network $N_2$ in \cref{fig:SNPR} the edge $g$ has $\delta(g) = 5$ descendants and the
root edge $e_\root$ has $\delta(e_\root) = 2n + 3r - 2 = 17$ descendants. We note that we do not
consider a vertex or an edge to be its own ancestor or descendant.

\paragraph{Network classes.}
A \phynet that has no reticulations is called a rooted binary phylogenetic tree, or in
this paper simply a \emph{\phytree}.
A \phynet in which every non-leaf vertex has a tree vertex as child is a \emph{tree-child} network. 
We denote by $\nets$ all \phynets with $n$ leaves, and with
$\trees$ and $\tchinets$ the subsets of $\nets$ consisting of all \phytrees and tree-child networks,
respectively.

One important well known property of tree-child networks is that each vertex $v$ contains a path to a
leaf consisting only of tree edges. Such a path is called a \emph{tree path} of $v$. Another
property is that a tree-child network has at most $n - 1$ reticulations \cite[Proposition 1]{CRV09}. 

A symmetry of a \phynet $N$ can be interpreted as an automorphism on $N$, distinct from the
identity function, that fixes the leaf set of $N$. The next proposition shows that tree-child
networks have no such symmetries. It is a reformulation of a result by McDiarmid
\etal~\cite[Lemma 5.1]{MSW15}. This implies that every vertex and every edge of a tree-child network is
uniquely identifiable, for example recursively by its set of descendant edges.

\pagebreak[4]
\begin{proposition}\label{thm:tchi:oneAutomorphism}
Let $N \in \tchinetsNoN$.

 \noindent
Then $N$ has exactly one automorphism that fixes its leaf set.
\end{proposition}
 
Note that it can also be shown that so-called normal networks, \tsib networks and level-1
networks without parallel edges have no such symmetries~\cite{SpOPhyN}. 
We will see why this is favourable for counting neighbours in the
next section, and discuss at the end, in \cref{sec:discussion}, why the problem gets harder
for more complex networks, which can have such symmetries.

\paragraph{Suboperations.}
To define SNPR operations we first have to define several suboperations. Let $G$ be a
directed acyclic graph. A degree-two vertex $v$ of $G$ with parent $u$ and child $w$ gets
\emph{suppressed} by deleting $v$ and the edges $(u,v)$ and $(v, w)$ and adding the edge $(u,
w)$. An edge $(u, w)$ of $G$ gets \emph{subdivided} by adding a new vertex $v$, deleting the edge
$(u, w)$ and adding the edges $(u, v)$ and $(v, w)$. Hence, a subdivision is the reverse of a
suppression.

Let $N$ be a \phynet. We say that an edge $e = (u, v)$ of $N$ with $u$ not a reticulation gets
\emph{pruned} by transforming it into the half edge $(., v)$ and suppressing $u$. In reverse, we say
a half edge $(., v)$ gets \emph{regrafted} to an edge $(u, w)$ by becoming the edge $(u', v)$ where
$u'$ is a new vertex subdividing $(u, w)$.

\paragraph{SNPR.}
Let $N \in \nets$. Let $e = (u, v)$ with $u$ not a reticulation and $f = (x, y)$ an edge
that is not a descendant edge of $e$. 
Then, like Bordewich \etal~\cite{BLS16}, we define the \emph{SubNet Prune and Regraft} (SNPR)
operation that transforms $N$ into a \phynet $N' \in \nets$ by applying exactly one of the
following three operations:
\begin{enumerate}[leftmargin=*,label=(SNPR+)]
  \item[(\SNPR)] If $f \neq e$, an \SNPR operation $(e, f)$ prunes $e$ and regrafts it to $f$.
  \item[(\SNPRP)] If $f \neq e$, an \SNPRP operation $(e, f)$ subdivides $f$ and $e$ with new
  vertices $u'$ and $v'$, respectively, and adds the edge $(u', v')$. If $e = f$, an \SNPRP
  operation $(e, f)$ subdivides $e$ twice with $u'$ and $v'$ such that $u'$ is parent of $v'$, and adds the edge
  $(u', v')$.
  \item[(\SNPRM)] If $e$ is a reticulation edge, an \SNPRM operation $(e)$ deletes $e$ and
  suppresses $u$ and $v$.
\end{enumerate}
For an \SNPR operation $(e, f)$ (or equivalently for other types of operations), we write $(e,
f)(N)$ to denote the \phynet $N'$ that results from applying $(e, f)$ to $N$. We note that \SNPR 
is for \rphytrees indeed a generalisation of rSPR. 
Bordewich \etal~\cite{BLS16} have shown that the three types of SNPR operations are
reversible. This means that for every \SNPR operation that transforms $N$ into $N'$, there exists an
\SNPR operation that transforms $N'$ into $N$, and that for every \SNPRP operation, there exists an
inverse \SNPRM operation, and vice versa. The SNPR operation induces thus a distance
function and a metric on $\nets$. 
Bordewich \etal~\cite{BLS16} also showed that $\nets, \trees$ and $\tchinets$ and other classes are
connected under SNPR. Moreover, they showed that the corresponding diameter of $\tchinets$ is linear
in $n$.

\begin{figure}[htb]
 \centering
 \includegraphics{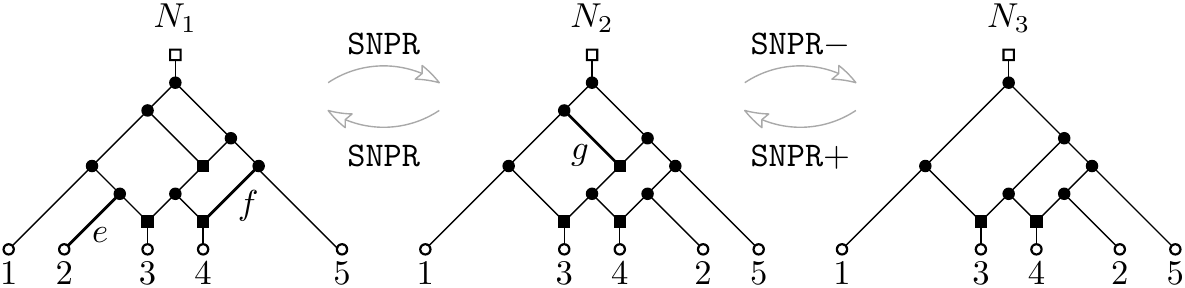}
 \caption{The \phynet $N_2$ can be obtained from $N_1$ by the \SNPR operation $(e,f)$.
 The \phynet $N_3$ can be obtained from $N_2$ with the \SNPRM operation $(g)$. Both operations have
 a corresponding \SNPR and \SNPRP operation, respectively, that reverses the transformation.}
 \label{fig:SNPR}
\end{figure}

\paragraph{NNI.}
The \emph{Nearest Neighbour Interchange} (NNI) operation was defined on (unrooted)
\phytrees~\cite{Rob71}, but has recently been generalised to unrooted and rooted \phynets by Huber
\etal~\cite{HLMW16,HMW16}. We define the NNI operation here only for tree-child networks, because we will
only use it as a tool to count special sets of SNPR operations. Furthermore, our notation and
explanation below differ from the one of Huber \etal~\cite{HLMW16,HMW16} for technical reasons and
since we consider rooted \phynets.

Let $N \in \tchinets$ and $e = (u, v)$ be an edge of $N$. Note that $e$ can not be a pure
reticulation edge, since $N$ is tree child. If $e$ is an inner edge, let $g \neq e$ be an edge
incident to $v$. If $u$ is a tree vertex, let $f$ be the sibling edge of $e$ and otherwise a parent
edge of $e$.
Then we define the \emph{Nearest Neighbour Interchange} (NNI) operation that transforms $N$ into
$N' \in \tchinets$ by applying exactly one of the following three operations:
\begin{enumerate}[leftmargin=*,label=(NNI+)]
  \item[(\NNI)]  If $e$ is an inner tree edge, an \NNI operation $(f, e, g)$ prunes $g$ and
  regrafts it to $f$, and, if $e$ is a reticulation edge, prunes $f$ and regrafts it to $g$.
  \item[(\NNIP)] An \NNIP operation $(e, f)$ subdivides $f$ with vertex $u'$ and $e$ with vertex
  $v'$, and adds the edge $(u', v')$.
  \item[(\NNIM)] If $e$ is the long side of a triangle (defined
  later), an \NNIM operation $(e)$ deletes $e$ and suppresses $u$ and $v$.
\end{enumerate}
The edge $e$ is called the \emph{axis} of the \NNI operation $(f, e, g)$. Note that our
definitions allow that $N' = N$. \cref{fig:NNI} illustrates the three types of NNI operations. We
note that all three types of NNI operations are special cases of SNPR operations. Furthermore,
like for SNPR, we observe that \NNI operations are reversible and that \NNIP and \NNIM operations
are mutually inverse.

\begin{figure}[htb]
 \centering
 \includegraphics{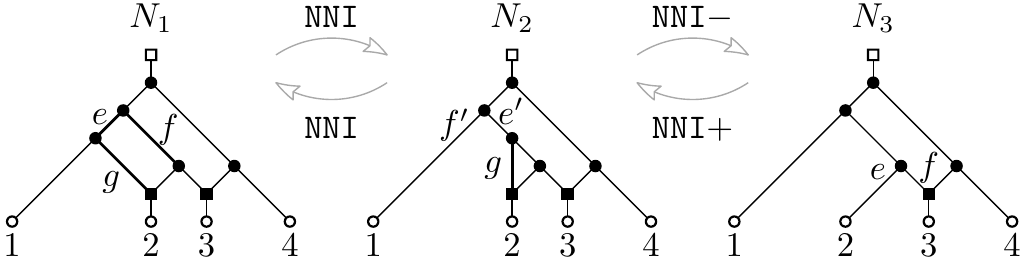}
 \caption{The \phynet $N_2$ can be obtained from $N_1$ by the \NNI operation $(f, e, g)$.
 $N_3$ can be obtained from $N_2$ with the \NNIM operation $(g)$. Both operations have a
 corresponding \NNI operation $(f', e', g)$ and \NNIP operation $(e, f)$, respectively, that
 reverses the transformation.} 
 \label{fig:NNI}
\end{figure}

\paragraph{Neighbourhood.}
Two \phynets that are one SNPR operation apart are called \emph{SNPR neighbours}.
The \emph{unit SNPR neighbourhood} $U_{\OPSNPR}(N)$ of $N$ is the set of all SNPR
neighbours of $N$. When we consider a tree-child network $N$, we are only
interested in neighbours that are also tree child. We call this
neighbourhood the \emph{unit SNPR tree-child neighbourhood} and denote it
by $U_{\OPSNPR}^{\tchinetsNoN}(N)$. \cref{fig:tchi:neighbourhood} gives an example of a tree-child
network and its unit tree-child SNPR neighbourhood.

For $N \in \nets$, we denote by $\Theta_{\OPSNPR}(N)$ the set of all SNPR operations on
$N$. (This can, in fact, be a multi-set, since $(e, f)$ can denote an \SNPR or an \SNPRP operation.)
An operation on a tree-child network that yields again a tree-child network is called
\emph{tree-child respecting}.
If $N \in \tchinets$, we write $\Theta_{\OPSNPR}^{\tchinetsNoN}(N)$ to denote the set of
\tchi-respecting SNPR operations on $N$. The definitions for $\trees$ and \SNPR, \SNPRP, and \SNPRM
are analogous.

An operation $\theta \in \Theta_{\OPSNPR}(N)$ is called \emph{trivial}, if $\theta(N) = N$.
Furthermore, we call two distinct operations on $N$ \emph{redundant}, if they yield
the same \phynet $N'$. A set of pairwise redundant operations is called a \emph{redundancy set}.
Note that $\abs{ \Theta_{\OPSNPR}(N) } \geq \abs{ U_{\OPSNPR}(N)}$, since
there can be trivial operations and redundancy sets.

\begin{figure}[htb]
 \centering
 \includegraphics[width=\textwidth]{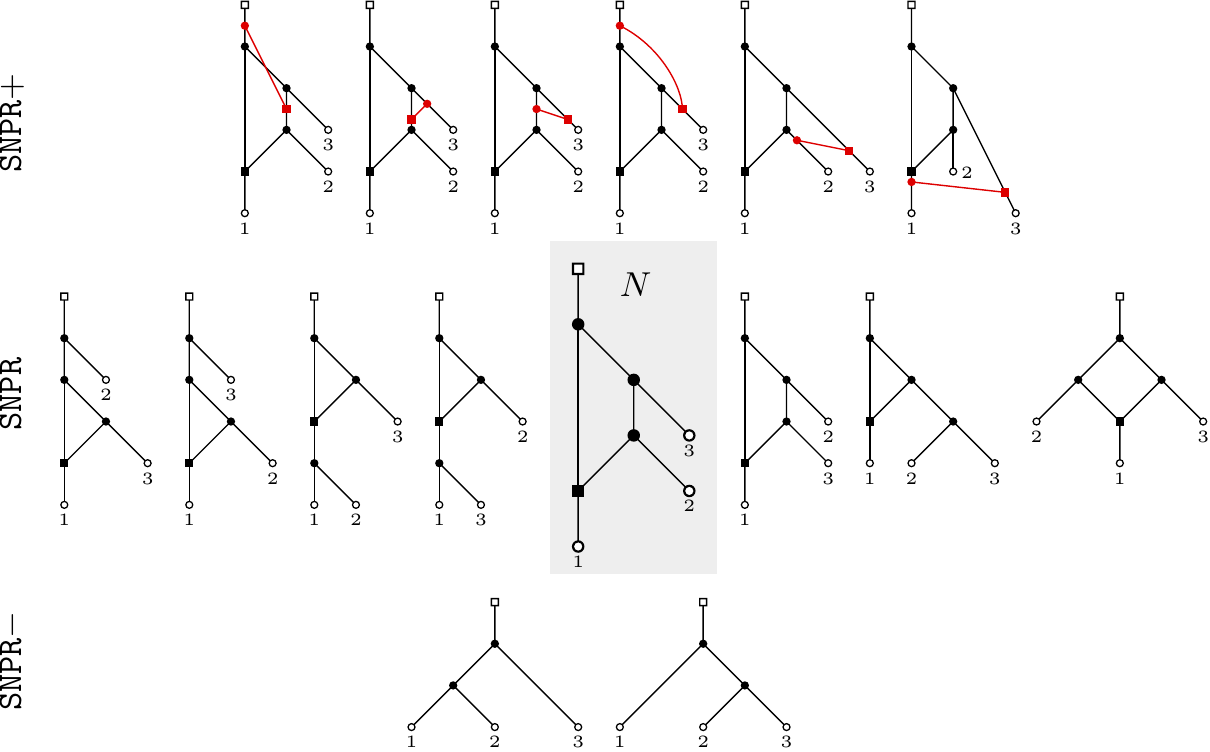}
 \caption{A tree-child network $N$ in the middle, with its SNPR tree-child neighbourhood
 around it. The top row consists of the \SNPRP neighbours, the middle row of the \SNPR neighbours,
 and the bottom row of the \SNPRM neighbours.} 
 \label{fig:tchi:neighbourhood}
\end{figure}

\paragraph{Structures.}
Let $N \in \tchinets$. In the following we define certain subgraphs of $N$, which we call
structures, that are the determining factor of whether operations on $N$ are
tree-child respecting, trivial and redundant. 
\cref{fig:tchi:neighbourhood:structures} accompanies our description of these structures.

\begin{figure}[htb]
 \centering
 \includegraphics[width=\linewidth]{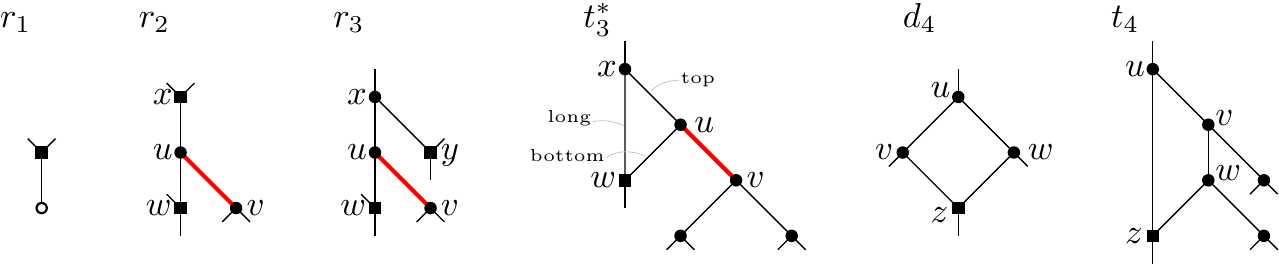}
 \caption{An $r_1$, $r_2$, and $r_3$ structure, a tree-branching triangle $\treebranchtriangles$,
 a diamond $\dias$, a trapezoid with outgoing tree edges $\trapis$. The critical edges (bold red)
 and the paths between the reticulations of the $r_2$ and $r_3$ structure are highlighted.}
 \label{fig:tchi:neighbourhood:structures}
\end{figure}

An $r_2$ structure of $N$ is a path of length two from a reticulation $x$ via a vertex $u$ to a
reticulation $w$. An $r_3$ structure consists of four vertices $x, y, u, w$ with edges $(x, y)$,
$(x, u)$ and $(u, w)$ where $y$ and $w$ are reticulations. 
We refer to the undirected path $w, u, x, y$ as the \emph{underlying path} of the
structure.
Note that in both an $r_2$ and an $r_3$ structure, since $N$ is tree child, both $u$ and its second
child $v$ are tree vertices. 
We abuse the notation to denote by $r_2$ and $r_3$ also the number of these structures in $N$. 
We define $r_1$ as the number of reticulations in $N$ whose child is a leaf. 

Note that an $r_3$ structure with $y = w$ is a triangle. Formally, a \emph{triangle} of $N$
consist of three vertices $x, u, w$ with the edges $(x, u)$, $(x, w)$ and $(u, w)$. We call the edge
$(x, u)$ the \emph{top side}, $(x, w)$ the \emph{long side}, and $(u, w)$ the \emph{bottom side} of
the triangle. 
We denote the number of triangles in $N$ by $\trias$. 
Furthermore, let $v \neq w$ be the second child of $u$. 
We note that since $N \in \tchinets$, we know that $u$ is a tree vertex.
Then if $v$ is incident to three pure tree edges, we call the triangle a \emph{tree-branching
triangle}.
We denote the number of those by $\treebranchtriangles$. 
We want to point out that every tree-branching triangle of $N$ is counted as an $r_3$ structure, as
a triangle and as a tree-branching triangle.

For an $r_2$, $r_3$ structure, or a triangle, with the notation from above, we call the tree edge
$(u, v)$ the \emph{critical} edge of this structure. See again
\cref{fig:tchi:neighbourhood:structures}, where the critical edges are highlighted, and note how
pruning them yields a vertex without a tree child. This will be important in the next section when
we consider tree-child respecting \SNPR operations.

A \emph{diamond} of $N$ is an undirected four-cycle consisting of edges $(u, v)$, $(u,
w)$, $(v, z)$ and $(w, z)$. 
A \emph{trapezoid} is an undirected four-cycle consisting of edges $(u, v)$, $(v, w)$, $(w, z)$ and
$(u, z)$. 
We note that in both cases $z$ is a reticulation. 
Important for us are trapezoids with the outgoing edges of the four-cycle at $v$ and $w$ being pure
tree edges.
We denote by $\dias$ the number of diamonds and with $\trapis$ the number of trapezoids with two 
outgoing pure tree edges.

%% file: sections/tchiSNPR.tex
\section{SNPR neighbourhood of a tree-child network}
\label{sec:neighbourhood:SNPR:tchi}

Throughout this section let $N \in \tchinets$. To count tree-child neighbours of $N$, it is
necessary to understand whether an SNPR operation on $N$ results again in a tree-child network, and
which operations on $N$ are redundant or trivial. In the following we show that this only depends on
different substructures of $N$. We consider \SNPR operations first, showing which respect the \tchi
property, which are trivial and which are redundant. This then allows us to count the number of
neighbours. After that, we include \SNPRP and \SNPRM operations to consider the SNPR
neighbourhood.

\paragraph{Tree-child respecting \SNPR operations.} % ----------------------------------------------
\pdfbookmark[2]{Tree-child respecting SNPR operations}{SNPRrespecting}
Let $\theta = (e, f) \in \Theta_{\SNPR}^{\tchinetsNoN}(N)$ and $e = (u, v)$. 
For $\theta$ to respect the tree-child property, neither pruning $(u,v)$ nor then regrafting $(.,v)$ to
$f$ can yield a non-leaf vertex without a tree child.
Roughly speaking and as we show in the following lemma, this implies that if $e$ is a
critical edge, then there are only limited options for $f$, and also that not both $e$ and $f$ can
be reticulation edges. \cref{fig:SNPR:tchiRespecting} illustrates the cases where $e$ is a critical
edge.

\begin{figure}[htb]
 \centering
 \includegraphics{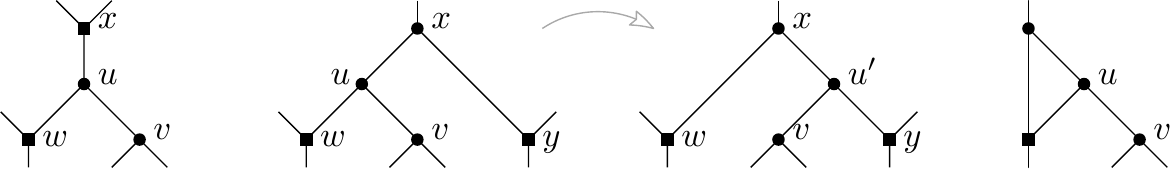}
 \caption{Illustration of the cases $r_2$, $r_3$, and triangle, where there are only trivial \SNPR
 operations that prune $(u, v)$ or in the case of $r_3$ exactly one non-trivial one as shown.
 This is formalised in \cref{lem:SNPR:tchiRespecting}}
 \label{fig:SNPR:tchiRespecting}
\end{figure}

\begin{lemma}\label{lem:SNPR:tchiRespecting} 
Let $N \in \tchinets$ and $(e, f) \in \Theta_{\SNPR}(N)$, $e = (u, v)$, $f = (x, y)$ not a
descendant of $e$.

	\noindent
Then $N' = (e, f)(N)$ is a tree-child network if and only if one of the following cases holds:
\begin{enumerate}[label=(\roman*)]
  \item $e$ is a reticulation edge and $f$ is not a reticulation edge; \label{c:tchiResp:retEdge}
  \item $e$ is a pure tree edge that is not critical.\label{c:tchiResp:normal}
  \item $e$ is a critical edge and $f$ is incident to $u$;\label{c:tchiResp:trivial}
  \item $e$ is a critical edge of an $r_3$ structure with underlying path $w, u, x, y$
  such that $u$ and $y$ are the children of $x$ and $f = (x, y)$; \label{c:tchiResp:rThree}
  \item $e$ is a critical edge of a triangle and $f$ is the long side of the triangle.
  \label{c:tchiResp:triangle}
\end{enumerate}
\end{lemma}
\begin{proof}
We prove this by considering the different types of $e = (u, v)$. By definition of an \SNPR
operation, $u$ can not be a reticulation. Thus, $e$ can not be a pure reticulation edge or an
impure tree edge. Let $e$ be an impure reticulation edge, i.e. let $v$ be a reticulation. Then,
since $N \in \tchinets$, the sibling $w$ of $v$ with shared parent $u$ is a tree vertex. Thus
after pruning $e$ and suppressing $u$, the parent of $u$ has in $N'$ the vertex $w$ as tree child.
Hence, a reticulation edge can always be pruned. Now, if $f = (x, y)$ is a reticulation edge, then
the new vertex $u'$ in $N'$, resulting from the subdivision of $f$, has the two children $v$ and
$y$, which are both reticulations. Thus, if $e$ is a reticulation edge, $f$ can not be a
reticulation edge. If $f$ is not a reticulation edge (\cref{c:tchiResp:retEdge}), then, in $N'$, the
new vertex $u'$ has the tree child $y$, the vertex $x$ has the tree child $u'$ and and all other
vertices stay unaffected.

Next, let $e$ be a tree edge. Clearly, $e$ is pure. 
If $e$ is not critical (\cref{c:tchiResp:normal}), then either the sibling of $v$ or the sibling of
$u$ is a tree vertex.
Without loss of generality let $w$, the sibling of $v$, be a tree vertex. Then, after
pruning $e$ and suppressing $u$, the parent $x$ of $u$ has $w$ as a tree child in $N'$. 
Since $v$ is a tree vertex, regrafting $(,v)$ to any edge $f$ does not create a non-leaf vertex
without tree child. Hence, $N'$ is tree child.

If $e$ is critical and $f$ incident to $e$ (\cref{c:tchiResp:trivial}), as $f$ is not a descendant
edge of $e$, then $N' = N$ and $N'$ is thus tree child.
If $e$ is the critical edge of an $r_2$ structure, then clearly $f$ being incident to $e$ is the
only option for $N'$ to be tree child. If $e$ is the critical edge of an $r_3$ structure, then after
pruning $e$ and suppressing $u$, the parent $x'$ of $u$ has the two reticulations $y'$ and $w$ as
children if and only if $e$ is not regrafted to an incident edge and if $f \neq (x', y')$
(\cref{c:tchiResp:rThree}).
In the case that the $r_3$ structure is a triangle, this yields that $f$ is the long side of the
triangle (\cref{c:tchiResp:triangle}). Since we covered all types of $e$, the described choices of
$e$ and $f$ cover all \tchi-respecting \SNPR operations.
\end{proof}

We now know when exactly an \SNPR operation respects the tree-child property. We can thus continue with
counting them. Let $E_R$ denote all reticulation edges, let $E_{T^*}$ denote all pure non-critical
tree edges and let $\delta_T$ be the restriction of the $\delta$ function that only counts
descendant edges that are tree edges.

\begin{lemma}\label{lem:tchi:SNPR:numOps}
Let $N \in \tchinets$.

	\noindent
Then the number of \tchi-respecting \SNPR operations on $N$ is
\begin{equation*}
\begin{split}
\abs{ \Theta_{\SNPR}^{\tchinetsNoN}(N) } 
 &= 4n^2 + 10nr - 2n(r_2 + r_3) - 6n + 2r^2 - 3r(r_2 + r_3) - 5r\\
 &\, + 4r_2 + 5r_3 + 2 - \sum\limits_{e \in E_{T^*}}\delta(e) - \sum\limits_{e \in E_R}
 \delta_T(e)\text{.}
\end{split}
\end{equation*}
\end{lemma}
\begin{proof}
	Following \cref{lem:SNPR:tchiRespecting}, we prove this by distinguishing different types of
	the pruned edge $e$. We use the fact that $N$ has $m = 2n + 3r - 1$ edges.
	First, any reticulation edge $e = (u, v)$ can be regrafted to any
	non-reticulation edge that is not descendant of $e$. 
	Hence, there are the following many such operations:
	\begin{equation}\label{eq:SNPR:tchiRespecting:retEdge}
	2r(m - 2r) - \sum_{e \in E_R} \delta_T(e) = 4nr + 2r^2 - 2r -
	\sum_{e \in E_R} \delta_T(e)
	\end{equation}
	\cref{eq:SNPR:tchiRespecting:retEdge} uses $\delta_T(e)$ instead of
	$\delta(e)$, since we would otherwise double count the forbidden operations of regrafting to an
	edge that is reticulation edge and descendant of the pruned edge. 
	
	If $e \in E_{T^*}$, i.e. a pure non-critical tree edge, then
	$e$ can be pruned and regrafted to every edge not $e$ itself or a descendant of $e$. 
 	Hence, there are the following many such operations:
	\begin{multline}\label{eq:SNPR:tchiRespecting:nonCritEdge}
	 (m - 3r - r_2 - r_3)(m - 1) - \sum_{e \in E_{T^*}} \delta(e)\\
 	= 4n^2 + 6nr - 2n(r_2 + r_3) - 6n - 3r(r_2 + r_3 + 1) + 2r_2 + 2r_3 + 2 -
 	\sum_{e \in E_{T^*}} \delta(e)
	\end{multline} 	
	
	If $e$ is the critical edge of an $r_2$ or $r_3$ structure (including triangles), then there
	are only $2$ or $3$ operations, respectively. Hence, there are the following many such
	operations:
	\begin{equation}\label{eq:SNPR:tchiRespecting:critEdge}
	2r_2 + 3r_3
	\end{equation}
 	Adding
 	\cref{eq:SNPR:tchiRespecting:retEdge,eq:SNPR:tchiRespecting:nonCritEdge,eq:SNPR:tchiRespecting:critEdge}
 	together, the lemma follows.
\end{proof}

\paragraph{Trivial \SNPR operations.}
\pdfbookmark[2]{Trivial SNPR operations}{SNPRtrivial}  
Pruning an edge and regrafting it at the same edge is a trivial \SNPR operation.
Another trivial \SNPR operation $(e, f)$ arises for every triangle where $e$ is its critical edge
and $f$ is its long side. Furthermore, the reticulation edges of a triangle induce a trivial
operation each, as the proof of the following lemma shows.

\begin{lemma}\label{lem:tchi:SNPR:trivial}
Let $N \in \tchinets$. 

	\noindent
Then there are $4n + 4r + 3\trias - 4$ trivial operations in
$\Theta_{\SNPR}^{\tchinetsNoN}(N)$.
\end{lemma}
\begin{proof}
	Let $(e,f) \in \Theta_{\SNPR}^{\tchinetsNoN}(N)$ with $e = (u, v)$ and $f = (x, y)$. An operation
	$(e,f)$ can be trivial in three ways. First, $f$ is incident to $e$ at $u$.
	The root edge and edges $(u, v)$ with a reticulation $u$ are not prunable.
	Therefore, there are $m - r - 1$ prunable edges and $2(m - r - 1)$ trivial \tchi-respecting
	\SNPR operations.
	
	Second, $f$ is isomorphic to the edge $g$ created by pruning $e$ and suppressing $u$.
	However, this can only happen if $f$ and $g$ are parallel edges, since by
	\cref{thm:tchi:oneAutomorphism} there are no pairs of isomorphic edges in a tree-child network. This
	means that the critical edge of a triangle gets pruned. Thus, there are $\trias$ many trivial \SNPR
	operations of that type.
  	
  	Third, let $f$ be neither of the above. Let the edges of $N$ be labelled and then in $N' = (e,
  	f)(N)$ let all labels be as in $N$ except those affected $(e, f)$. Let the
  	regrafted edge have the label $e$. Then, since $N' = N$, there has to be an edge $e' = (u', v')$
  	in $N'$ that is, without label, the same edge as $e$ in $N$. By the choice of $f$, this can not
  	be $e$. The edges $e$ and $e'$ got, so to say, swapped. Then, since by
  	\cref{thm:tchi:oneAutomorphism} every vertex is unique, $v = v'$ follows.
  	The edges $e$ and $e'$ are thus reticulation edges. For $N' = N$, clearly, $e$ and $e'$ have to be
  	the reticulation edges of a triangle: If we prune the long side of a triangle and regraft it to
  	the critical edge of the triangle, it results again in $N$. An equivalent operation exists for
  	the bottom edge of the triangle. Hence, there are $2\trias$ such trivial \SNPR operations.
  	
	Furthermore, these three cases do not overlap and we thus counted all trivial \SNPR operations
	on $N$. Since $2(m - r - 1) + 3\trias = 4n + 4r + 3\trias - 4$, the lemma follows.
\end{proof}

\paragraph{Redundant \SNPR operations.}
\pdfbookmark[2]{Redundant SNPR operations}{SNPRredundancy}  
We now consider when and how non-trivial \SNPR operation on tree-child networks can be redundant. 
Humphries and Wu~\cite{HW13} used NNI operations to count redundancies of SPR and TBR
operations on unrooted trees. The following lemma states how \SNPR and \NNI correspond to each other with regards
to redundancy on rooted trees.

\begin{lemma}\label{lem:threeSPRoneNNI}
Let $T \in \trees$, and let $\theta, \theta' \in \Theta_{\SNPR}^{\treesNoN}(T)$, $\theta
\neq \theta'$ be redundant with $\theta(T) = \theta'(T) = T' \neq T$.

	\noindent
Then there exists an \NNI operation $\sigma \in \Theta_{\NNI}^{\treesNoN}(T)$ such that
$\sigma(T) = T'$.
Furthermore, every redundancy set of $\Theta_{\SNPR}^{\treesNoN}(T)$ has size three.
\end{lemma} 

We prove \cref{lem:threeSPRoneNNI}, after we generalise observations on how an \SNPR redundancy
can occur. \cref{fig:SNPR:redundancy:pureTreeEdge} illustrates how, in the lemma, three \SNPR
operations correlate to an \NNI operation. There, in the subgraph of $T$, the axis $(x, u)$ of the
\NNI operation is a pure inner tree edge. We observe that, therefore, such an \NNI operation can
also induce a redundancy in a \phynet. The edge $(x, u)$ further has the siblings $v$ and $w$ as children
of $u$ and their uncle $y$ as child of $x$. Now, the three redundant \SNPR operations could be
described as follows. First, $((u, v), (x, y))$ prunes $v$ and regrafts it as sibling of $y$.
Second, $((x, y), (u, v))$ prunes $y$ and regrafts it as sibling of $v$. Third,  $((u, w), (p, x))$
prunes $w$ and regrafts it above $x$, thus makes $v$ and $y$ siblings. 
In general, to find redundancies of \SNPR operations, we can fix two vertices that stand in a
certain relation in $N'$, but not yet in $N$. Then, to create this relation, say making $v$ and $y$
siblings, we can either regraft one as sibling of the other or alter the path between them. 
We formalise this with the following lemma, after we precisely describe the initial situation.

\begin{figure}[htb]
% \begin{figure}[htb]
%  \centering
%  \includegraphics{NNI3SPR}
%  \caption{Illustration of \cref{lem:threeSPRoneNNI}, showing that an \NNI operation induces a
%  redundancy set of of three \rSPR operations.}
%  \label{fig:NNI3SPR}
% \end{figure}
 \centering
 \includegraphics{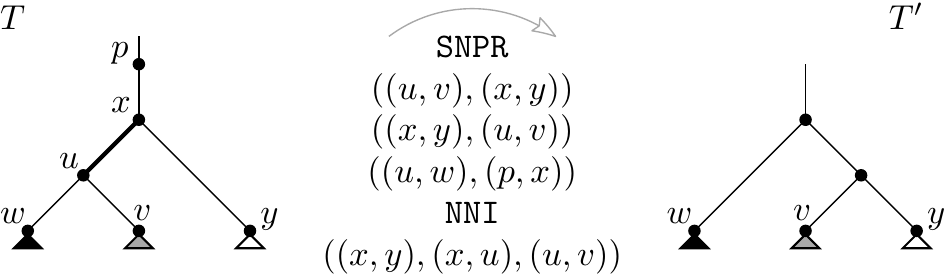}
 \caption{Correlation of an \NNI operation with a pure inner tree edge as axis and three \SNPR
 operation, all being pairwise redundant.}
 \label{fig:SNPR:redundancy:pureTreeEdge}
\end{figure}

Let $N, N' \in \tchinets$ be neighbours with $N' = \theta(N), \theta \in
\Theta_{\SNPR}^{\tchinetsNoN}(N)$. Let the vertices in both $N$ and $N'$ be labelled and let
$\theta$ preserve these labels, except, of course for removed or new vertices. 
Let now $v$ and $y$ be distinct vertices with the same labels, and such that neither is ancestor of
the other in both $N$ and $N'$. 
We now say that $v$ and $y$ are in a \emph{desired relation} if one of the following holds:
\begin{itemize}
  \item The vertex $v$ is a sibling, an uncle or a nephew of $y$ in $N'$ via a path $P'$, but $v$ is
  in a different relation to $y$ in $N$.
  \item The vertex $v$ is an uncle or a nephew of $y$ in $N'$ via a path $P'$ and in $N$ via a path
  $P \neq P'$.
\end{itemize}
In the second condition, $P \neq P'$ means that the labels of the vertices on $P$ differ for a least
one vertex from the labels of the vertices on $P'$.

\begin{lemma}\label{lem:makeSiblings} 
Let $N, N' \in \tchinets$ and $\theta \in \Theta_{\OPSNPR}^{\tchinetsNoN}(N)$ with $\theta(N) =
N' \neq N$. Let $v$ and $y$ be in a desired relation via the path $P'$ in $N'$.

	\noindent
Then there are only the following possibilities of how $\theta$ operates on $N$ to yield $N'$: 
\begin{enumerate}[label=(\roman*)]
  \item an incoming edge of $v$ or $y$ gets pruned and regrafted such that $v$ becomes 
  sibling, uncle or nephew, respectively, of $y$;
  \item an incoming edge of the parent of $v$ or $y$ gets pruned and regrafted such that $v$
  becomes uncle or nephew, respectively, of $y$;
  \item an edge $e = (u, w)$ with $u$, but not $w$, being on a path connecting $v$ and $y$, gets
  pruned yielding $P'$;
  \item an edge gets regrafted to a path connecting $v$ and $y$ yielding $P'$.
\end{enumerate}
\end{lemma}
\begin{proof}
	The existence of $v$ and $y$ in $N'$ after applying $\theta$ to $N$ means that
	$\theta$ does not prune an outgoing edge of $v$ or $y$.
	For $\theta$ to yield the desired relation and path $P'$ in $N'$, $\theta$ can either alter
	an existing path between $v$ and $y$ by one vertex, i.e. (iii) or (iv), or prune an edge of an
	existing path between $v$ and $y$ and regraft it such that a desired path $P'$ gets created, i.e.
	(i) or (ii).
\end{proof}

Applying \cref{lem:makeSiblings} means that we can consider an \SNPR operation $\theta$, find two
vertices $v$ and $y$ in the resulting network $N'$ that are in desired relation, and then check
whether other \SNPR operations corresponding to one of the possibilities listed in the lemma exist
that are redundant to $\theta$. We can now prove \cref{lem:threeSPRoneNNI}.

\begin{proofof}[\cref{lem:threeSPRoneNNI}]
The lemma states that if two non-trivial \SNPR operations $\theta$ and $\theta'$ are redundant on a
\phytree, that there is an \NNI operation that is redundant to them. 
Let $v$ and $y$ be two distinct vertices of $T$ that are not siblings, but that are under preserving
of labels by $\theta$ siblings in $T'$. Then $v$ and $y$ are in a desired relation. 
By \cref{lem:makeSiblings} follows then that $\theta$ corresponds to an $\NNI$ operation
and, moreover, that $\theta'$ is one of two other \SNPR operations redundant to $\theta$.
Hence, the redundancy set containing $\theta$ has size three.
\end{proofof}

We now count the number of \tchi-respecting \SNPR operations that we can discard due to
redundancy. In the proof for the following proposition, we will see that redundancies only arise
from a few different sources, like \NNI operations with the axis being a pure inner tree edge.
The other sources of redundancies are operations that create a triangle, the reticulation edges
of a triangle (see \cref{fig:tchi:SNPR:redundancy:triangles}), and the existence of tree-branching
triangles, diamonds and $\trapis$ trapezoids (see
\cref{fig:tchi:SNPR:redundancy:triangleDiamond,fig:tchi:SNPR:redundancy:triangleTrapezoid,fig:tchi:SNPR:redundancy:diamondTrapezoid}).
We note that an \NNI operation on a \phynet with the axis being a reticulation edge or an impure
tree edge does not correspond to a redundancy set of \SNPR operations. 

\begin{figure}[htb]
\begin{center}
  \includegraphics{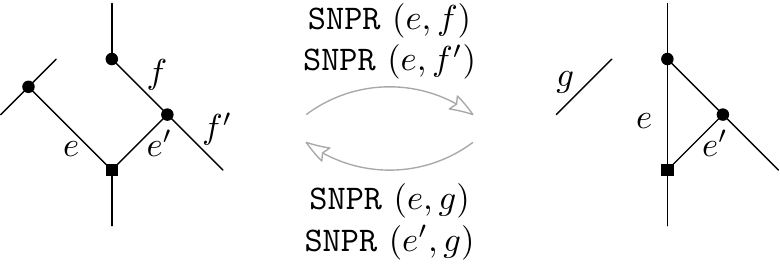}
  \caption{Redundancies due to the creation of a triangle and the reticulation edges of a triangle
  with listed redundant \SNPR operations.}
  \label{fig:tchi:SNPR:redundancy:triangles}
\end{center}
\end{figure}

\begin{proposition}\label{lem:tchi:SNPR:overcountingOps}
Let $N \in \tchinets$.

	\noindent
Then the number of non-trivial redundant \SNPR operations of $\Theta_{\SNPR}^{\tchinetsNoN}(N)$
minus the number of redundancy sets of non-trivial \SNPR operations of
$\Theta_{\SNPR}^{\tchinetsNoN}(N)$ is 
$$2n(2 + \trias) + r(2 + \trias) + 4r_1 - 2r_3 - 8\trias + \treebranchtriangles + 3\dias + \trapis -
8 - \sum\limits_{e \in \treebottomedges} \delta_T(e)\text{.}$$
\end{proposition}
\begin{proof}
Let $\theta = (e, f) \in \Theta_{\SNPR}^{\tchinetsNoN}(N)$ such that $\theta(N)= N' \neq N$.
Let $e = (u, v), f = (x, y)$. 

The operations we want to count are those that we want to \emph{discard} when counting neighbours. 
To count all the operations we can discard, we go through the sources of redundancy one
by one and determine the sizes of the corresponding redundancy sets. 
In order to find all sources, the idea of this proof is to consider when and where $\theta$ can be
redundant with other operations. To cope with all possibilities, we fix a reticulation including a
cycle for which this reticulation is the lowest vertex (i.e. descendant of all other vertices on the
cycle). Then, when applying $\theta$, it can be distinguished whether the size of this cycle gets
decreased, increased, or whether only the order of edges with start vertex on the cycle gets altered.
Therefore in the following case distinction, we denote the change from a cycle of size $c$ to size
$c'$ by $[c \to c']$. Note that there can be no cycle of size 1 or 2. A cycle size of 0 denotes
either no cycle, i.e. a tree, or that no cycle is under consideration or of any influence to
redundancies of $\theta$.

One source of redundancy that we already identified are \NNI operations with a pure inner tree edge
as axis. A \phynet has $n + r_1 - 2$ pure inner tree edges (all edges minus any incident to
reticulations, leaves or the root) each inducing two \NNI operations. However, if the axis is part
of an $r_3$ structure, then one of the possible \NNI for this axis is not tree-child respecting. 
Also, if it is part of a triangle, the operation is either
trivial or not tree-child respecting. There are thus $2(n + r_1 - 2) - r_3 - \trias$ \NNI operations
of interest, each with an \SNPR redundancy set of size three. 
Therefore, we discard the following many non-trivial \tchi-respecting \SNPR operations:

\begin{equation} \label{eq:redundancy:NNI}
	4n + 4r_1 - 2r_3 - 2\trias - 8
\end{equation}

We freely use \cref{thm:tchi:oneAutomorphism} throughout the remainder of this proof.
\begin{description}
  \item[$\lbrack 0 \to 0 \rbrack$] 
  	If no cycle is involved, the part where the \SNPR operations make changes is tree-like and there
	are thus only tree edges. It follows thus from \cref{lem:threeSPRoneNNI} that the redundancy comes
  	from an \NNI operation. Hence, these redundancies are covered by \cref{eq:redundancy:NNI}.

  \item[$\lbrack 3 \to 3 \rbrack$]
  	A triangle with fixed reticulation has only one shape and thus can not be transformed into
  	another one with a single \SNPR operation.
\end{description}
In the following two cases, we will see redundancies due to the reticulation edges of triangles. We
will count the \SNPR operations we discard afterwards. 
\begin{description}
  \item[$\lbrack 3 \to 4 \rbrack$]
  	A triangle can be transformed into a cycle of size four either by pruning one of its edges and
  	regrafting it to an edge outside of the triangle, thus including this edge as third outgoing edge
  	or by regrafting an edge from outside to the triangle. We will see that considering only the
  	latter case will also cover all of the former case.
  	
  	Let the edge $e = (u, v)$ have distance at least two to the triangle and let $f = (x, y)$ be an
  	edge of the triangle. Assuming $e$ has distance greater than two, it is clear (for example
  	with the analysis of \cref{lem:makeSiblings}) that there can only be a redundancy, if $e$ is the
  	reticulation edge of another triangle and $f$ the top side of the triangle. However, no \SNPR
  	operation pruning an edge of the fixed triangle or incident to it can be redundant to this
  	operation and thus any redundancy would be accredited to the other triangle. Therefore, assuming
  	now that $e$ has distance two to the triangle, the following cases can be distinguished.
  	\begin{description}
  	  \item[{\normalfont (i)} $u$ is parent of triangle,] \textbf{$f$ is
  	  long side of triangle.} Requiring that $e$ is a tree edge, the triangle gets transformed into a diamond.
  		Using the analysis of \cref{lem:makeSiblings} with $y$ and $v$ as siblings in $N'$ yields that there are
  		four redundant \SNPR operations, as illustrated
  		by \cref{fig:tchi:SNPR:redundancy:triangleDiamond}.
  		Three \SNPR operations can be associated to the \NNI operation $(e, c, f)$ where $c$ is the
  		incoming edge of the triangle. Also, pruning one of the two reticulation edges of the triangle
  		and regrafting it to $e$ is redundant to doing the same with the other. We note that the
  		\SNPR operation $(f, e)$ corresponds to both redundancies.
  		
  	  \item[{\normalfont (ii)} $u$ is sibling of reticulation of triangle,] \textbf{$f$ is long side of
  	  triangle.} Requiring that both $e$ and its sibling edge are tree edges (and thus
  	  that the triangle is a tree-branching triangle), this transforms the triangle again into a diamond (see
  	  	again \cref{fig:tchi:SNPR:redundancy:triangleDiamond}). This time the analysis, again with $y$
  	 	and $v$ as siblings, yields a redundancy set of size two, namely regrafting $e$ and its
  	  	sibling edge $e'$ to $f$. 

		This means that each tree-branching triangle of $N$ induces a redundancy set of size two. We thus
		discard one \SNPR operation per such triangle:
		\begin{equation} \label{eq:redundancy:treebranchtriangle}
			\treebranchtriangles
		\end{equation}
  	  	
  	  \item[{\normalfont (iii)} $u$ is parent of triangle,] \textbf{$f$ is top side of triangle.}
  	  	Without a requirement on $e$, this transforms the triangle into a trapezoid (see
  	  	\cref{fig:tchi:SNPR:redundancy:triangleTrapezoid}). Like in (i), the analysis with $v$ being
  	  	uncle of $y$, yields again an \NNI operation redundancy with an overlap of a triangle
  	  	reticulation edges redundancy. Furthermore, this can coincide
  	  	with a transformation of another triangle into a trapezoid of the next case.
  	  	
  	  \item[{\normalfont (iv)} $u$ is sibling of reticulation of triangle,] \textbf{$f$ is top side
  	  of triangle.} This requires that the sibling edge of $e$ is a tree edge and transforms
  	  the triangle into a trapezoid (see again \cref{fig:tchi:SNPR:redundancy:triangleTrapezoid}). 
  	   	The analysis yields the same as in the previous case. If the sibling edge of $e$ is not a
  	   	tree edge, we would have an $r_3$ structure and  $N'$ would not be tree child. 
  	   
  	   	Furthermore, if $e$ is a tree edge, the case is equivalent to $f$ being the bottom side of
  	   	the triangle.
  	  \item[{\normalfont (v)} $u$ is parent of triangle,] \textbf{$f$ is bottom side of triangle.}
  	  	The analysis yields that there is no redundancy of \SNPR operations here.
	\end{description}

\begin{figure}[htb]
\begin{center}
  \includegraphics[width=\linewidth]{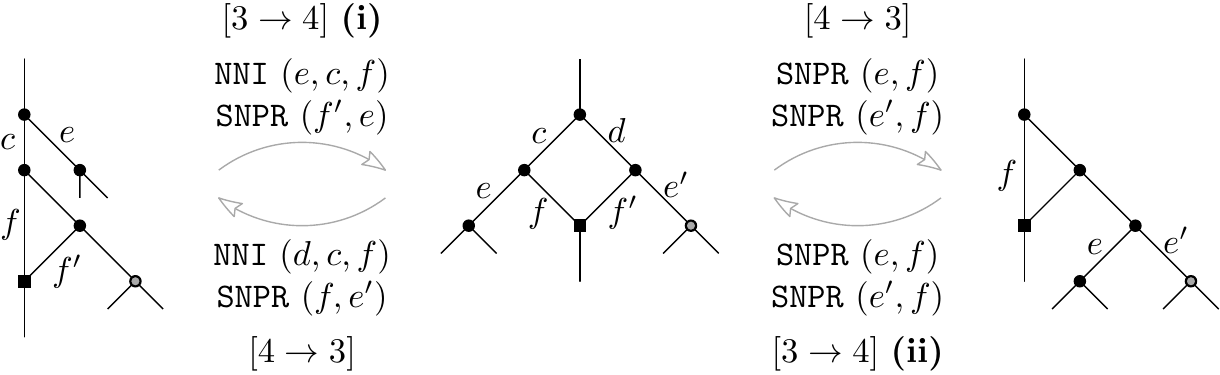}
  \caption{Transformation of triangles into diamonds and vice versa with listed redundancies,
  covering parts of the cases $[3 \to 4]$ and $[4 \to 3]$. }
  \label{fig:tchi:SNPR:redundancy:triangleDiamond}
\end{center}
\end{figure}
\begin{figure}[htb]
\begin{center}
  \includegraphics[width=\linewidth]{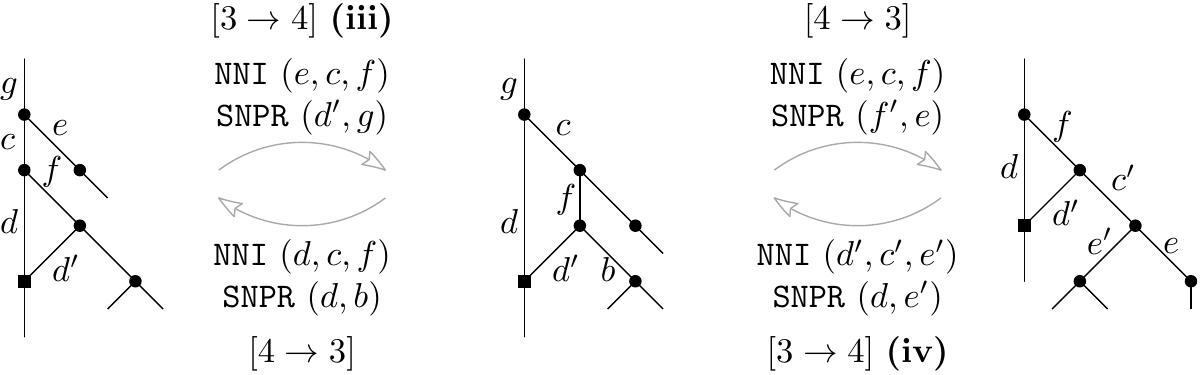}
  \caption{Transformation of triangles into trapezoids and vice versa with listed redundancies,
  covering parts of the cases $[3 \to 4]$ and $[4 \to 3]$.} 
  \label{fig:tchi:SNPR:redundancy:triangleTrapezoid}
\end{center}
\end{figure}

  \item[$\lbrack 3 \to c, c \geq 5 \rbrack$]
	Since it is not possible to add two outgoing edges to a triangle by regrafting them to the triangle
	with a single \SNPR operation, the size can only be increased by pruning an edge of the triangle
	and regrafting it to an edge $f$ at the desired distance. This yields the same neighbour for the two
	reticulation edges, but different cycles for the top side of the triangle and one of its
	reticulation edges.
\end{description}

From the last two cases, we know that two \SNPR operations $\theta$ and $\theta'$ that prune the two
different reticulation edges of a triangle and regraft it to the same edge are always redundant. In
cases, where $\theta$ and $\theta'$ are also redundant to \SNPR operations corresponding to an \NNI
operation, either $\theta$ or $\theta'$ also corresponds to that \NNI operation. In any case,
without loss of generality, we can discard all (non-trivial) \SNPR operations that prune an
edge $e \in \treebottomedges$, i.e the bottom side of a triangle:
\begin{equation} \label{eq:redundancy:triangles}
	2n\trias + r\trias - 4\trias - \sum\limits_{e \in \treebottomedges} \delta_T(e)
\end{equation} 

\begin{description}
  \item[$\lbrack 4 \to 3 \rbrack$]
	This is basically the analysis of $[3 \to 4]$ backwards. See again
	\cref{fig:tchi:SNPR:redundancy:triangleDiamond,fig:tchi:SNPR:redundancy:triangleTrapezoid} for
	illustrations.

  \item[$\lbrack c \to 3, c \geq 5 \rbrack$]
  	To create a triangle with a specific reticulation, one way is to prune one of its reticulation
  	edges and to regraft it to an edge incident to the other reticulation edge. This is the
  	reverse of $[3 \to c]$ and yields two redundancy sets of size two. 
  	
  	The second possibility is to to prune the parent edge of one of the reticulation edges and
  	regraft it to the other reticulation edge. Again, as seen in $[3 \to c]$, this is not redundant
  	to the other way or other operations.
\end{description}
The last two cases covered the creation of triangles. With \cref{eq:redundancy:treebranchtriangle}
we accounted for redundancies from a tree-branching triangle to a diamond. With
\cref{eq:redundancy:diamond} we do the reverse:
\begin{equation} \label{eq:redundancy:diamond}
	\dias
\end{equation} 
Furthermore, as the reverse of \cref{eq:redundancy:triangles}, each reticulation edge that is not
part of a triangle corresponds to two redundant \SNPR operations. In the case that a four cycle is
created, this can coincide with a redundancy due to an \NNI operation. Like before, we can discard
one of the two operations:
\begin{equation} \label{eq:redundancy:toTriangle}
	2r - 2\trias
\end{equation}

\begin{description}
  \item[$\lbrack 4 \to 4 \rbrack$]
	To change a cycle of size four into another cycle of size four, one can either change the order of
	the outgoing edges of a trapezoid, which is then equivalent to Case $[0 \to 0]$, or
	transform a trapezoid into a diamond or vice versa (see
	\cref{fig:tchi:SNPR:redundancy:diamondTrapezoid}). Applying \cref{lem:makeSiblings} for any of the
	directions with two appropriate vertices yields redundancy sets of size four. We see that three
	edges correspond to an \NNI operation. We have thus already counted a neighbour and can discard
	the fourth \SNPR operation. We note that there are two different transformations from a diamond to
	a trapezoid distinguished by the order of the resulting trapezoid's outgoing edges. Hence, we discard
	the following many \SNPR operations:
	\begin{equation} \label{eq:redundancy:fourFour}
		2\dias + \trapis
	\end{equation}
	
\begin{figure}[htb]
\begin{center}
  \includegraphics{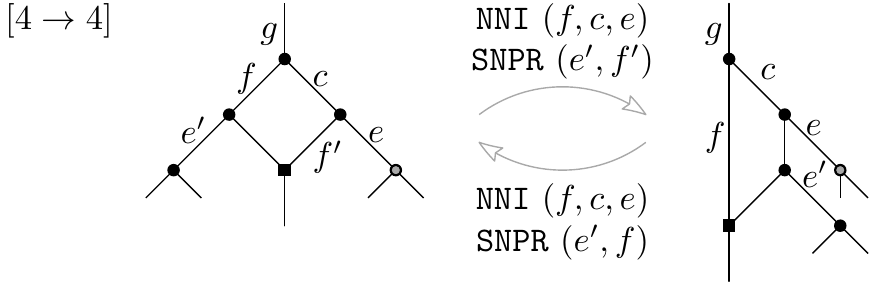}
  \caption{Transformation of diamonds into $\trapis$ trapezoids and vice versa with listed redundancies,
  illustrating the case $[4 \to 4]$.} 
  \label{fig:tchi:SNPR:redundancy:diamondTrapezoid}
\end{center}
\end{figure}

  \item[$\lbrack c \to c + 1, c \geq 4 \rbrack$]
	Adding a branch to a cycle of size at least four, and thus increasing its size by one, is, by using 
	\cref{lem:makeSiblings}, only possible if the operations correspond to an \NNI operation.

  \item[$\lbrack c \to c + x, c \geq 4, x \geq 2 \rbrack$]
  	Unlike in the case $[3 \to 5]$ there is obviously no redundancy of any edges of the cycle
  	anymore.

  \item[$\lbrack c + 1 \to c, c \geq 4 \rbrack$]
	This is the reverse of $[c \to c + 1]$ and there are thus only redundancies due to an
	\NNI operation.

  \item[$\lbrack c + x \to c, c \geq 4, x \geq 2 \rbrack$]
	As the reverse of the case $[c + x \to c]$, there are no redundancy in this case.  
\end{description}

We have now covered all cases of transformations of a cycle into another one. We have further
identified the different sources of redundancies and discarded non-trivial \tchi-respecting \SNPR
operations accordingly. Adding
\cref{eq:redundancy:NNI,eq:redundancy:treebranchtriangle,eq:redundancy:triangles,eq:redundancy:diamond,eq:redundancy:toTriangle,eq:redundancy:toTriangle,eq:redundancy:fourFour}
together, the statement follows. 
\end{proof}

\paragraph{The unit \SNPR tree-child neighbourhood.}
\pdfbookmark[2]{Unit SNPR tree-child neighbourhood}{SNPRNneighbourhood}
The unit \SNPR tree-child neighbourhood of $N$ is determined by the number of \tchi-respecting
\SNPR operations on $N$ (\cref{lem:tchi:SNPR:numOps}), from which trivial operations are subtracted
(\cref{lem:tchi:SNPR:trivial}) and operations that yield redundant neighbours are discarded
(\cref{lem:tchi:SNPR:overcountingOps}).

\begin{theorem}\label{thm:tchi:SNPR:neighbourhood}
Let $N \in \tchinets, n \geq 2$.

	\noindent
Then the unit \SNPR tree-child neighbourhood $U_{\SNPR}^{\tchinetsNoN}(N)$ of $N$ has size 
\begin{equation*}
\begin{split}
	\abs{ U_{\SNPR}^{\tchinetsNoN}(N) } 
	  & = 4n^2 + 10nr - 2n(r_2 + r_3 + \trias) - 14n
		+ 2r^2 - r(3r_2 + 3r_3 + \trias)\\
	  &  - 11r - 4r_1 + 4r_2 + 7r_3 + 5\trias 
		- \treebranchtriangles - 3\dias - \trapis
		+ 14\\ 
	  & - \sum\limits_{e \in E_{T^*}} \delta(e)
		- \sum\limits_{e \in E_R} \delta_T(e) 
	    + \sum\limits_{e \in \treebottomedges} \delta_T(e) \text{.}
\end{split}
\end{equation*}
\end{theorem}

% > example
\cref{tbl:tchi:neighbourhoodExample:parameters} lists the values for the parameters of the \tchi
network $N$ from \cref{fig:tchi:neighbourhood}. Applying these values to
\cref{thm:tchi:SNPR:neighbourhood} we get that $N$ has $7$ different \SNPR tree-child neighbours, as
depicted in \cref{fig:tchi:neighbourhood}.

\begin{table}[htbp]
	\begin{tabular}{c|l|l}
    parameter & description & in $N$\\
    \hline
    $n$ & \# leaves & 3\\
    \hline
    $r$ & \# reticulations & 1\\
    \hline
    $r_1$ & \# reticulations with leaf as child & 1\\
    \hline
    $r_2$ & \# $r_2$ structures & 0\\
    \hline
    $r_3$ & \# $r_3$ structures & 0\\
    \hline
    $\trias$ & \# triangles & 0\\
    \hline
    $\treebranchtriangles$ & \# tree-branching triangles & 0\\
    \hline
    $\dias$ & \# diamonds & 0\\
    \hline
    $\trapis$ & \# $\trapis$ trapezoids & 1\\
    \hline
    $\sum_{e \in E_{T^*}} \delta(e)$ & \# descendant edges of pure non-critical tree edges & 13\\
    \hline
    $\sum_{e \in E_R} \delta_T(e)$ & \# descendant tree edges of reticulation edges & 2\\
    \hline
    $\sum_{e \in \treebottomedges} \delta_T(e)$ & \# descendant tree edges of triangle bottom sides
    & 0\\
	\end{tabular}
	\caption{Parameters for the network $N$ from \cref{fig:tchi:neighbourhood}, which has a unit
	\SNPR tree-child neighbourhood of size $7$.}
	\label{tbl:tchi:neighbourhoodExample:parameters}
\end{table} 

If $N$ is a \phytree, it has of course zero reticulations and no special structures. We thus get the
following formula for \phytrees.

\begin{corollary}
Let $T = (V, E) \in \trees$.

	\noindent
Then the unit \SNPR tree neighbourhood $U_{\SNPR}^{\treesNoN}(T)$ of $T$ has size
$$ \abs{ U_{\SNPR}^{\treesNoN}(T)} = 4n^2 - 14n + 14 - \sum\limits_{e \in E} \delta(e)
\text{.}$$
\end{corollary}

Accounting for the fact that we count descendants of edges and not ancestors of vertices, this
formula equals the result by Song~\cite{Son03}.

\paragraph{The unit SNPR$+$ and SNPR$-$ tree-child neighbourhood.} % -------------------------------
\pdfbookmark[2]{Unit SNPRP and SNPRM tree-child neighbourhood}{SNPRPneighbourhood}
We now consider \SNPRP and \SNPRM operations and count again first the number of such operations.
For this, let $E_{PS}$ denote the set of edges that are pure tree edges with a sibling pure tree
edge.

\begin{lemma}\label{lem:SNPRP:numOps}
Let $N \in \tchinets$ with $n \geq 2$. 

	\noindent
Then
\begin{equation*}
\begin{split}
\abs{ \Theta_{\SNPRP}^{\tchinetsNoN}(N) } 
  &= 4n^2 - 2nr - 8n - 2r^2  + 2r + 4 - \sum\limits_{e \in E_{PS}} \delta_T(e) \text{, and}\\
\abs{ \Theta_{\SNPRM}^{\tchinetsNoN}(N) } 
  &= 2r  \text{.}
\end{split}
\end{equation*}
\end{lemma}
\begin{proof}
For an \SNPRP operation $(e, f)$, which adds an edge from $f$ to $e$, for $(e, f)(N)$ to be tree
child, $e$ has to be a pure tree edge with a sibling pure tree edge.
This implies that $e \neq e_{\root}$ and that $e$ can not be incident to a reticulation or sibling
edge of a reticulation edge.
Otherwise, if $e$ would be incident to a reticulation, this would yield a pure reticulation edge,
and if $e$ would be the sibling edge of a reticulation edge, this would yield a vertex with two
reticulations as children. In either case $(e, f)(N)$ would not be tree child.
Thus every reticulation induces a set of five edges, consisting of the three edges incident to it
and their two sibling edges. Since $N \in \tchinets$, clearly these sets are disjoint for every
pair of reticulations of $N$. There are thus $m - 5r - 1$ choices for $e$.

Next, by the definition of an \SNPRP operation, the edge $f$ can not be a descendant of $e$.
Furthermore, $f$ can not be a reticulation edge or $e$.
Otherwise, if $f$ would be a reticulation edge, this would yield a vertex with two reticulations as
children, and if $f = e$, the operation would create a parallel edge. 
In either case $(e, f)(N)$ would not be tree child. Clearly any other choice of $f$ is fine.
For any feasible choice of $e$, there are thus $m - 2r - 1 - \delta_T(e)$ choices of $f$.
With the $\delta_T(e)$ summing up to $\sum_{e \in E_{PS}} \delta_T(e)$ over all choices of $e$ and 
$(m - 5r - 1)(m - 2r - 1) = 4n^2 - 2nr - 8n - 2r^2  + 2r + 4$ the first statement follows.

Concerning $\Theta_{\SNPRM}^{\tchinetsNoN}(N)$, it is easy to see that removing any reticulation
edge of a tree-child network yields again a tree-child network. There are thus $2r$ tree-child
respecting \SNPRM operations on $N$.
\end{proof}

Note that \SNPRP and \SNPRM operations are never trivial, since they change the number of
reticulations. However, for both of these types of operations redundancies might exist. Like for
most \SNPR redundancies, these redundancies are equivalent to \NNIP and \NNIM operations, as the
proof of the following proposition shows.

\begin{proposition}\label{lem:tchi:SNPRP:neighbourhood}
Let  $N \in \tchinets$ with $n \geq 2$.

	\noindent
Then the unit \SNPRP tree-child neighbourhood $U_{\SNPRP}^{\tchinetsNoN}$ of $N$ has size 
 	$$ \abs{ U_{\SNPRP}^{\tchinetsNoN}(N) } = 
 	4n^2 - 2nr - 10n - 2r^2  + 4r + 6 - \sum\limits_{e \in E_{PS}} \delta_T(e) \text{,}$$
and the unit \SNPRM tree-child neighbourhood $U_{\SNPRM}^{\tchinetsNoN}$ of $N$ has size
 	$$ \abs{ U_{\SNPRM}^{\tchinetsNoN}(N) } = 
 	2r - \trias \text{.}$$
\end{proposition}  
\begin{proof}
	This proof uses the concept of the proof of \cref{lem:tchi:SNPR:overcountingOps}. 
	For the first part, we assume that the considered \SNPRP operations are tree-child respecting.
	\begin{description}
	  \item[$\lbrack 0 \to 3 \rbrack$]
	  	Let $f = (u, v)$ be a tree edge with $v$ having two outgoing pure tree edges $e = (v, w)$ and
	  	$e' = (v, y)$. Then a reticulation and a triangle can be added by the \SNPRP operation $(e, f)$,
		which adds an edge from $f$ to $e$. This is however redundant to the \SNPRP operation $(e, e')$.
	  	It follows by the uniqueness of $e$ that there are no further redundant \SNPRP operations.
	  	Furthermore, these operations are redundant to the \NNIP operation $(e, e')$.
	  	  
	  \item[$\lbrack 0 \to c, c \geq 4 \rbrack$]
	  	Similar to $[0 \to 3]$, the edge $e$ that gets subdivided for the new reticulation is unique.
	  	However, for a new cycle of size at least 4, there are no two edges that can be chosen
	  	interchangeably to be subdivided for the source of the new reticulation edge to yield the same
	  	network $N'$.
	\end{description}
	There are $n - r - 1$ pairs of siblings of pure tree edges in $N$. To account for redundancy, we
  	can thus discard $2(n - r - 1)$ \SNPRP operations. The first part follows then from
  	\cref{lem:SNPRP:numOps}.

	\begin{description}
	  \item[$\lbrack 3 \to 0 \rbrack$]
	  	If a triangle gets removed, removing one of the reticulation edge of the triangle is redundant to
	  	removing the other. Since no reticulations are isomorphic in $N$, there can be no further
	  	reticulation edges in $N$ that if removed would yield the same network. These \SNPRM operations
	  	are thus equivalent to an \NNIM operation of the respective triangle.
	
	  \item[$\lbrack c \to 0, c \geq 4 \rbrack$]
	  	The reticulation edges of a cycle of size at least four are neither isomorphic nor can
	  	they change roles like in triangles. Thus removing a reticulation edge can not
	  	be redundant to removing another of a cycle of size at least four.
	\end{description}
	There are $\trias$ many \NNIM operations in $N$. Discarding one \SNPRM operation per triangle, the
	second part follows again from \cref{lem:SNPRP:numOps}.
\end{proof}

We can again consider the tree-child network $N$ from \cref{fig:tchi:neighbourhood}. Since $N$ has one
reticulation, but no triangles, $N$ has two \SNPRM neighbours. Using
\cref{tbl:tchi:neighbourhoodExample:parameters} and the fact that $\sum_{e \in E_{PS}}
\delta_T(e) = 2$, we get that $N$ has $6$ \SNPRP tree-child neighbours.

\paragraph{The unit SNPR tree-child neighbourhood.} % ---------------------------------------------------
\pdfbookmark[2]{Unit SNPR tree-child neighbourhood}{SNPRneighbourhood}
To obtain the total size of the SNPR tree-child neighbourhood of a tree-child network $N$ we can now add
together the sizes of the \SNPR, \SNPRP and \SNPRM neighbourhoods
(\cref{thm:tchi:SNPR:neighbourhood,lem:tchi:SNPRP:neighbourhood}).

\begin{theorem}\label{thm:tchi:SNPRneighbourhood}
Let $N \in \tchinets$ with $n \geq 2$.

	\noindent
Then the unit SNPR tree-child neighbourhood $U_{\OPSNPR}^{\tchinetsNoN}$ of $N$ has size 
\begin{equation*}
\begin{split}
	\abs{ U_{\OPSNPR}^{\tchinetsNoN}(N) } 
	  & = 8n^2 + 8nr - 2n(r_2 + r_3 + \trias) - 24n
		- r(3r_2 + 3r_3 + \trias) - 5r\\
	  & - 4r_1 + 4r_2 + 7r_3 + 4\trias - \treebranchtriangles - 3\dias - \trapis	+ 20\\ 
	  & - \sum\limits_{e \in E_{T^*}} \delta(e)
		- \sum\limits_{e \in E_R} \delta_T(e) 
		- \sum\limits_{e \in E_{PS}} \delta_T(e)
	    + \sum\limits_{e \in \treebottomedges} \delta_T(e)\text{.}
\end{split}
\end{equation*}
\end{theorem}

We conclude this section with the comment that, since each parameter and sum can be computed in
$\Oh(n)$ time ($r, m \in \Oh(n)$), the unit SNPR tree-child neighbourhood size of a tree-child network $N$
can be computed in linear time $\Oh(n)$.

%% file: sections/extremeValues.tex
\section{Minimal and maximal neighbourhoods}
\label{sec:minMaxValues}

The formula for the unit SNPR tree-child neighbourhood of a tree-child network depends on a lot of
parameters. It is therefore of interest to see how small and big a neighbourhood can get in terms
of $n$.

\begin{proposition}
Let $n \geq 2$. Then
\begin{equation*}
\begin{split}
	            n - 1 \leq &\min_{N \in \tchinets} \{ \abs{ U_{\OPSNPR}^{\tchinetsNoN}(N) } \}  
	             	\leq \frac{3}{2}n^2 - \frac{7}{2}n + 2\text{, and}\\
	8n^2 - \Oh(n \log_2 n) \leq & \max_{N \in \tchinets} \{ \abs{ U_{\OPSNPR}^{\tchinetsNoN}(N) } \}  
					< 16n^2 - 38n + 26\text{.}
\end{split}
\end{equation*}
\end{proposition}
\begin{proof}
% < min
We first establish a lower bound for the minimal neighbourhood size of a tree-child network. 
Let $N \in \tchinets$ with $n \geq 2$ and $r$ reticulations. 
By \cref{lem:tchi:SNPRP:neighbourhood}, we have that 
$\abs{ U_{\SNPRM}^{\tchinetsNoN}(N) } = 2r - \trias$. Each reticulation gives rise to two different
\SNPRM operations with redundancy sets of size at most two.
Furthermore, a reticulation edge can be added from the root edge $e_{\root}$ to every other pure
tree edge that is not sibling edge of a reticulation edge. There are $2n - 2r - 2$ such edges. Note
that $\trias \leq n -1$. There are thus at least $2n - 2r - 2 + 2r - \trias = 2n - \trias - 2 \geq
n - 1$ SNPR tree-child neighbours of $N$. This is sharp for a tree-child network with $n = 2$ and $r = 1$.

% min <
Next, we look at an upper bound for the minimal neighbourhood size of a tree-child network.
For this we consider a family of tree-child networks, where each has a relative small neighbourhood. Let
$N_r \in \tchinets, n \geq 2$ be a chain of $r$ triangles, where one triangle is child of the
triangle above it in the chain, like $N_4$ in \cref{fig:tchi:neighbourhood:extremeValues}.
Since $N_r$ has $n - 1$ reticulations, $N_r$ has no \SNPRP neighbours.
Removing a reticulation edge from one of the triangles corresponds to one of $r = n - 1$ different
\SNPRM neighbours of $N_r$.
Concerning \SNPR operations, the only prunable non-critical edges are the long sides of the
triangles (and the bottom sides, which however behave redundantly and are thus ignored). Since
these are reticulation edges they can, when pruned, only be regrafted to tree edges that are
not their descendants. For the long side of the triangle closest to the root three
such edges exist, which however correspond to trivial operations. For the long side of the triangle
below six such edges exist, of which again three yield trivial operation. Thus, in total
there are $$3 + 6 + 9 + 12 + \ldots + (n - 2)3 = \sum_{i = 1}^{n -2}3i = \frac{3}{2}(n - 2)(n - 1) =
\frac{3}{2}n^2 - \frac{9}{2}n + 3$$ \SNPR neighbours. 
All together, $N_r$ has $\frac{3}{2}n^2 - \frac{7}{2}n + 2$ SNPR neighbours.

\begin{figure}[htb]
 \centering
 \includegraphics{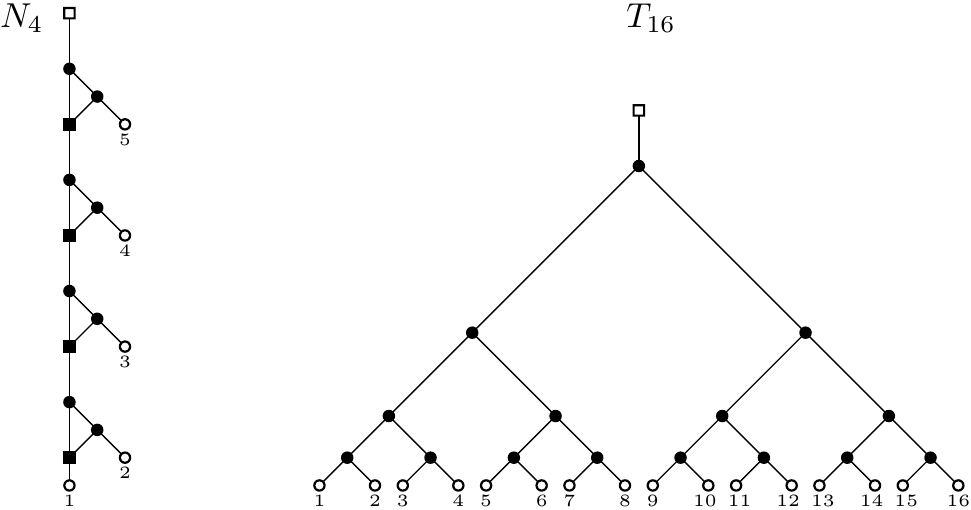}
 \caption{Tree-child networks with small and big neighbourhoods: The tree-child network $N_4$ has
 %$\frac{3}{2}n^2 - \frac{7}{2}n + 2$
 $\frac{3}{2}5^2 - \frac{7}{2}5 + 2 = 22$ SNPR neighbours.
 The balanced tree $T_{16}$ has at least 
%  $8n^2 - 4n \log_2 n - 22n + 22$ 
 $8 \cdot 16^2 - 4 \cdot 16 \log_2 16 - 18 \cdot 16 + 14 = 1518$ 
 SNPR neighbours.}
 \label{fig:tchi:neighbourhood:extremeValues}
\end{figure}

% < max
For a lower bound of the maximal SNPR tree-child neighbourhood size of a tree-child network, we consider the
balanced tree $T_n$ on $n$ leaves, as illustrated by $T_{16}$ in
\cref{fig:tchi:neighbourhood:extremeValues}. 
The formula for the SNPR neighbourhood (\cref{thm:tchi:SNPRneighbourhood}) is then $$8n^2
- 24n + 20 - \sum_{e \in E} \delta(e) - \sum_{e \in E_{PS}} \delta_T(e)\text{.}$$ 
For simplicity, we now assume that $n = 2^k$ with $k \geq 1$.
Then, for the first sum we have $$\sum_{e \in E} \delta(e) 
	= \sum_{i = 1}^{\log_2 n} i 2^i 
	= 2n \log_2 n - 2n + 2\text{.}$$
The second sum only differs from the first by the fact that the root edge $e_\root$ is not in $E_{PS}$
and thus 
$$\sum_{e \in E_{PS}} \delta_T(e) 
	= 2n \log_2 n - 2n + 2 - \delta(e_\root) 
	= 2n \log_2 n - 4n + 4\text{.}$$
In total this yields that, for $n = 2^k$, $T_n$ has $8n^2 - 4n \log_2 n - 18n + 14$ SNPR neighbours.
With no restriction on $n$, the tree $T_n$ has at least $8n^2 - \Oh(n \log_2 n)$ SNPR neighbours.

% max <
For an upper bound of the maximal \SNPR neighbourhood size, we estimate bounds for the various
parameters. 
We thus assume that the parameters $r_1, r_2, r_3, \trias, \dias$, and $\trapis$ are zero.
Concerning the sums $- \sum_{e \in E_{T^*}} \delta(e) - \sum_{e \in E_R} \delta_T(e) - \sum_{e \in
E_{PS}} \delta_T(e) + \sum_{e \in E_{\bar{b}_3}} \delta_T(e)$ of the SNPR neighbourhood formula, we
observe that the first and third sum are at best zero, that the second sum is at best $2r$ and that
the last sum is at best half of the second. The sums account therefore only for $-r$ in our
estimate. Assuming that $r = n - 1$ and thus maximal, we get that for any tree-child network $N$ the unit
SNPR tree-child neighbourhood has size at most $16n^2 - 38n + 26$.
\end{proof}

%% file: sections/discussion.tex
\section{Discussion and outlook}
\label{sec:discussion}

In this paper, we have presented formulas for the unit SNPR \tchi neighbourhood size of \tchi
networks.
We have shown that the neighbourhood size does not only depend on the number of leaves and
reticulations, but also on the shape of the network. In the formulas the shape is represented by the
occurrences of certain structures, like triangles and diamonds, and the number of descendant edges
of certain edge sets. The size of the \SNPRM neighbourhood is at most $n - 1$, because the number of
reticulations in a \tchi network is at most \mbox{$n - 1$}~\cite[Proposition 1]{CRV09}. On the other
hand, the \SNPR and \SNPRP neighbourhoods can both have a quadratic size in terms of $n$.
We presented further bounds on the minimal and maximal neighbourhood size.

The main tool in our proofs of redundancy of operations was \cref{thm:tchi:oneAutomorphism}, which
states every \tchi network has exactly one automorphism that fixes its leaf set. 
This allowed us to pinpoint
redundancies to some simple structures. Our methodology can be applied to other network classes,
especially subclasses of \tchi networks.
For example, normal networks are \tchi networks that do not contain both an edge $(u, v)$ and a path
from $u$ to $v$ that consists of at least two edges. Vertices and edges are thus also unique in
normal networks and, moreover, they do not contain triangles or trapezoids. 
While this simplifies counting \SNPR redundancies, it comes at the cost that counting all
normal-respecting \SNPR operations get harder. 
Another class that is suitable for our methodology is the class of level-1 networks. A \phynet is
level-1 if each of its biconnected components contains at most one reticulation. A level-1 network
without parallel edges is thus also a \tchi network. 
The size of the unit SNPR neighbourhood for normal and level-1 networks can be found in the
author's thesis~\cite{SpOPhyN}.

\begin{figure}[htb]
 \centering
 \includegraphics{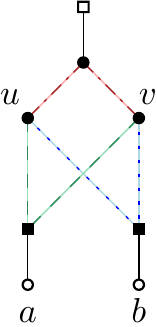}
 \caption{A \retvis (and \tbased) network with non-unique vertices and edges. The vertices $u$ and
 $v$ are isomorphic and so are the three groups of edges in same colour and style.}
 \label{fig:symmetries}
\end{figure}

Finding a formula for the unit SNPR neighbourhood is more complicated for the network classes
of \retvis and \tbased networks. 
A \phynet is \emph{reticulation visible} if every reticulation separates the root from at least one
leaf~\cite{HK07}. Roughly speaking, a \phynet is \emph{tree based} if it is contains a \phytree on
the same leaf set that covers all its vertices~\cite{FS15}.  
The class of \tchi networks is known to be a subclass of both the classes of \retvis and \tbased
networks~\cite{FS15}. 
However, \cref{fig:symmetries} shows a \retvis (and thus also \tbased)
network with only two leaves and two reticulations, where two isomorphic vertices and three
pairs of isomorphic edges exist. More complicated \phynets with bigger sets of pairwise isomorphic
vertices or edges can easily be found. Isomorphic edges in a \phynet have the consequence that
pruning one of them or regrafting to one of them is redundant to any of the others. Determining the
size of a neighbourhood would thus require to identify and account for all equivalences.
While we think that a recursive approach (similar to the work of Song~\cite{Son03,Son06}) to count
neighbours might still be possible, a closed formula is likely to require even more parameters.

The space of \rphynets under NNI has, in contrast to the unrooted case, not found much attention
in the literature so far. We defined NNI operations for \tchi networks as a tool to count several
redundancies. However, to consider the space itself and its properties could be of interest.